\documentclass[12pt,reqno]{amsart}

\usepackage[usenames]{color}
\usepackage{xcolor}
\usepackage{amssymb}
\usepackage{amsmath}
\usepackage{amsthm}
\usepackage{amsfonts}
\usepackage{amscd}
\usepackage{graphicx}
\usepackage{mathrsfs}
\allowdisplaybreaks[4]

\numberwithin{equation}{section}
\usepackage{graphicx}
\usepackage{ifpdf,epstopdf}
\usepackage{float}
\usepackage{caption}

\usepackage[pagebackref]{hyperref}
\usepackage{hyperref}
\hypersetup{
colorlinks=true,
linkcolor=red,
filecolor=brown,
citecolor={blue}
}
\newtheorem{thm}{Theorem}[section]
\newtheorem{lem}[thm]{Lemma}
\newtheorem{coro}[thm]{Corollary}
\newtheorem{prop}[thm]{Proposition}
\newtheorem{conj}[thm]{Conjecture}

\theoremstyle{definition}

\newtheorem{ex}[thm]{Example}

\newtheorem{rem}[thm]{Remark}

\numberwithin{equation}{section}

\newcommand{\pf}{{\rm PF}}
\newcommand{\tp}{{\rm TP}}

\def\R{\mathbb{R}}
\def\N{\mathbb{N}}

\begin{document}
\title
{Real-rootedness  of the type A minuscule polynomials}
\author{Ming-Jian Ding}
\address[Ming-Jian Ding]{School of Mathematical Sciences,
 Dalian University of Technology, Dalian 116024, P. R. China}
\email{ding-mj@hotmail.com}

\author[Jiang Zeng]{Jiang Zeng}
\address[Jiang Zeng]{Universite Claude Bernard Lyon 1, ICJ UMR5208, CNRS, Centrale Lyon, INSA Lyon, Universit\'e Jean Monnet\\
69622, Villeurbanne Cedex, France}
\email{zeng@math.univ-lyon1.fr}
\date{\today}

\begin{abstract}
We prove two recent conjectures of Bourn and Erickson (2023) regarding the real-rootedness
of a certain family of polynomials $N_n(t)$ as well as  the sum of their coefficients.
These polynomials   arise as the numerators of generating functions in the context of the discrete one-dimensional earth mover's distance (EMD) and
have also connection to the Wiener index of minuscule lattices.
We also  prove that the coefficients of $N_n(x)$ are asymptotically normal,
the coefficient matrix of $N_n(x)$ is  totally positive and the polynomial sequence
$N_n(x)$'s is $x$-log-concave.

\end{abstract}
\maketitle
%\tableofcontents
\section{Introduction}

For $n\in \N$ define  the polynomials
\begin{equation}\label{poly+exp+formula}
N_n(x):=\frac{1}{4n+2}\sum_{k=1}^{n-1}k(n-k)\binom{2n+2}{2k+1}x^k.
%=\sum_{k=1}^{n-1}N_{n,k}x^k
\end{equation}
The coefficients of $N_n(x)$ are  the Wiener index of a minuscule lattice of type~A~\cite{DFNW23}, i.e., the Hasse diagram of the poset of order ideals in a $k\times(n-k)$ rectangle,
we will refer to $N_n(x)$ as the {\it Minuscule polynomial of type A}.
Below we list these polynomials for the first few values of $n$:
\begin{align*}
N_1(x)&=0,\\
N_2(x)&=2x,\\
N_3(x)&=8x+8x^2,\\
N_4(x)&=20x+56x^2+20x^3,\\
N_5(x)&=40x+216x^2+216x^3+40x^4,\\
N_6(x)&=70x+616x^2+1188x^3+616x^4+70x^5.
\end{align*}
These polynomials  also appeared   as the numerators of generating functions in the context of the discrete one-dimensional earth mover's distance (EMD). Recently,
Bourn and Erickson~\cite{BE23} proved a conjecture of Bourn and Willenbring~\cite{BW20}  regarding the palindromicity and unimodality of $N_n(x)$ and proposed the following conjectures.

\begin{conj}\label{conj+poly+RZ}
The polynomial $N_n(x)$ has only real zeros for each $n \in \mathbb{N}$.
\end{conj}
\begin{conj}\label{conj+poly+iden}
For all positive integers $n$, we have $N_n(1)=(n-1)2^{2n-3}$.
\end{conj}

This paper was motivated by the above conjectures.
In Section~\ref{sec+poly+RZ}, we shall prove Conjecture \ref{conj+poly+RZ},
which implies that the polynomial $N_n(x)$ is log-concave and $\gamma$-positive.
In Section~\ref{sec+iden+poly}, we establish a polynomial identity for $N_n(x^2)$, 
which reduces directly to Conjecture~\ref{conj+poly+iden} for $x=1$.
By applying the latter identity
we show that the coefficients of $N_n(x)$ are asymptotically normal.
In Section~\ref{sec+log+concave}, we prove that the polynomial sequence $(N_n(x))_{n\geq0}$ is $x$-log-concave, 
i.e., $N^2_n(x)-N_{n+1}(x)N_{n-1}(x)\in \N[x]$ for $n\geq 2$. 
Lastly, we prove that the coefficient matrix $[N_{n,k}]_{n,k\geq0}$ is totally positive in Section~\ref{sec+TP}.

For reader's convenience we recall some terminology for the later use.
To a sequence of real numbers $(f_0, f_1, \ldots, f_n)$, 
we define the polynomial $f(x)=\sum_{k=0}^n f_k x^k $ and say that the polynomial $f(x)$ is
\begin{itemize}
\item {\it palindromic} (with respect to $n$) if $f_k=f_{n-k}$ for $0\leq k\leq n$,
\item \emph{unimodal} if there exists an index $m$ such that
\begin{equation*}
f_0\le f_1\le\cdots\le f_m\ge\cdots\ge f_{n-1}\ge f_n,
\end{equation*}
\item  \emph{log-concave} if $f_k^2 \geq f_{k-1}f_{k+1}$
for all $1 \leq k \leq n-1$.
\end{itemize}
It is known that the log-concavity of a positive sequence implies its unimodality~\cite{Br15}.
Let $f(x), g(x) \in \mathbb{R}\left[x \right]$ be real-rooted
with zeros $\left\{r_i\right\}$ and $\left\{s_j\right\}$, respectively.
We say that $g(x)$ \emph{interlaces} $f(x)$ if $\deg(f(x))=\deg(g(x))+1=n$ and
\begin{equation}\label{1}
	r_{n} \leq s_{n-1}\leq \cdots  \leq s_{2} \leq r_{2} \leq s_{1} \leq r_{1},
\end{equation}
and that $g(x)$ \emph{alternates left of} $f(x)$ if $\deg(f(x))=\deg(g(x))=n$ and
\begin{equation}\label{2}
	s_{n} \leq r_{n} \leq \cdots  \leq s_{2} \leq r_{2} \leq s_{1} \leq r_{1}.
\end{equation}
We write $g(x) \preceq f(x)$ if  $g(x)$ \emph{interlaces} $f(x)$ or $g(x)$ \emph{alternates left of} $f(x)$.

\section{Real-rootedness of polynomials $N_n(x)$}\label{sec+poly+RZ}

A finite sequence $\lambda=(\lambda_k)_{k=0}^{n}$ of real numbers is called a
\emph{multiplier $n$-sequence} if for any real-rooted polynomial $a_0 + a_1x + \cdots + a_n x^n$ of degree at most $n$,
the polynomial $a_0 \lambda_0 + a_1 \lambda_1 x + \cdots + a_n \lambda_n x^n$ has only real zeros.
The following  is a classical algebraic characterization for multiplier $n$-sequence.
\begin{lem}\label{lem+alg+MS+n}\cite[Theorem 3.7]{CC77}
A real sequence $(\lambda_k)_{k=0}^n$ is a multiplier $n$-sequence if and only if
the polynomial
\begin{equation*}
\sum_{k=0}^n\binom{n}{k}\lambda_kx^k
\end{equation*}
has only real zeros with same sign.
\end{lem}

\begin{ex}
From the binomial formula
$(1+x)^n=\sum_{k=0}^n {n\choose k} x^k$ we get successively
\begin{subequations}
\begin{align}
nx(1+x)^{n-1}&=\sum_{k=0}^n k{n\choose k} x^k,\\
nx(1+nx)(1+x)^{n-2}&=\sum_{k=0}^n k^2{n\choose k} x^k,
\end{align}
and
\begin{equation}
n(n-1)x(1+x)^{n-2}=\sum_{k=0}^n k(n-k){n\choose k} x^k.
\end{equation}
\end{subequations}
By the above lemma, the sequence $(k(n-k))_{0\leq k\leq n}$ is a multiplier $n$-sequence.
\end{ex}
The \href{https://mathworld.wolfram.com/ChebyshevPolynomialoftheSecondKind.html}
{Chebyshev polynomials of the second kind} $U_n(x)$ can be defined by the generating function
\begin{equation*}
\sum_{n=0}^{\infty}U_n(x)t^n=\frac{1}{1-2xt+t^2}.
\end{equation*}
%for $|x| \leq 1$ and $|t| \leq 1$.
They also have the following explicit formulae \cite[p. 696]{Zw95}
\begin{subequations}
\begin{align}
U_n(x)&=\sum_{k=0}^{\left \lfloor n/2 \right \rfloor}\binom{n+1}{2k+1}x^{n-2k}(x^2-1)^k\label{chebyshev+poly+sum+formula}\\
&=2^{n}\prod_{k=1}^n\left[x-\cos\left(\frac{k\pi}{n+1}\right)\right]\label{chebyshev+poly+formula}.
\end{align}
\end{subequations}

\begin{thm}\label{thm+minu+poly+RZ}
The polynomial $N_n(x)$ has only real zeros for each positive integer $n$.
Moreover, we have $N_n(x) \preceq N_{n+1}(x)$ for any $n \in \mathbb{N}$.
\end{thm}
\begin{proof}
By formula~\eqref{chebyshev+poly+sum+formula},
\begin{subequations}
\begin{equation}\label{chebyshev+poly+2n+1}
  U_{2n+1}(x)= x^{2n+1}
 \sum_{k=0}^{n}\binom{2n+2}{2k+1}\left(\frac{x^2-1}{x^2}\right)^k.
\end{equation}
For  $n+2\leq k\leq 2n+1$,  let  $k=2n+2-\ell$, then $\ell \in [n]$ and
formula \eqref{chebyshev+poly+formula} implies
\begin{align}
U_{2n+1}(x) &= \prod_{k=1}^{n}\left[x-\cos\left(\frac{k\pi}{2n+2}\right)\right]\cdot x\cdot2^{2n+1}
               \cdot\prod_{\ell=1}^{n}\left[x-\cos\left(\frac{(2n+2-\ell)\pi}{2n+2}\right)\right]\nonumber\\
            &=x\cdot2^{2n+1}\prod_{k=1}^{n}\left[x^2-\cos^2\left(\frac{k\pi}{2n+2}\right)\right]\label{chebyshev+poly+2n+1+AB}.
\end{align}
Substituting $x^2 \rightarrow 1/(1-x)$ in $x^{-2n-1}U_{2n+1}(x)$ and
combining \eqref{chebyshev+poly+2n+1} and \eqref{chebyshev+poly+2n+1+AB}, we get
\begin{equation}\label{eq+dec+zero}
  f_n(x):= \sum_{k=0}^{n}\binom{2n+2}{2k+1}x^k
 = 2^{2n+1}\prod_{k=1}^{n}\left[1-(1-x)\cos^2\left(\frac{k\pi}{2n+2}\right)\right],
\end{equation}
\end{subequations}
which proves the real-rootedness of the polynomial $f_n(x)$.
By the above example, the polynomial $N_n(x)$ has only real zeros.

By \eqref{eq+dec+zero}, it is not difficult to see that $f_n(x) \preceq f_{n+1}(x)$ since the following inequalities hold:
$$
1-\frac{1}{\cos^2\left(\frac{(k+1)\pi}{2n+4}\right)}
\leq
1-\frac{1}{\cos^2\left(\frac{k\pi}{2n+2}\right)}
\leq 1-\frac{1}{\cos^2\left(\frac{k\pi}{2n+4}\right)}
$$
for each $k \in [n]$.
Note the fact that if $f(x) \preceq g(x)$, then $xf^{'}(x) \preceq xg^{'}(x)$, see \cite[Theorem 6.3.8]{RS02}.
Additionally, the relation $f(x) \preceq g(x)$ with $\deg(f(x))=\deg(g(x))-1=n$
implies that $x^nf(1/x) \preceq x^{n+1}g(1/x)$.
Due to $f_n(x) \preceq f_{n+1}(x)$,
it is routine to verify that $N_n(x) \preceq N_{n+1}(x)$ by the facts which are mentioned before.
\end{proof}

\begin{rem}
In view of Hermite-Biehler Theorem,
the real-rootedness of polynomial $f_n(x)$ follows also from the Hurwitz stability of $(1+x)^{2n+2}$,
see \cite[Theorem 6.3.4]{RS02} for more details.
\end{rem}
It is known~\cite[Lemma 7.1.1]{Br15} that if a polynomial with non-negative coefficients is real-rooted,
then it is log-concave, and if a positive sequence is log-concave, then it is unimodal.
Thus, Theorem~\ref{thm+minu+poly+RZ} implies the following result,
which generalizes the unimodality of $N_n(x)$~\cite[Corollary 4.2]{BE23}.

\begin{coro}
The polynomial $N_n(x)$ is log-concave for each $n \in \mathbb{N}$.
\end{coro}

By \eqref{poly+exp+formula} it is obvious that $N_n(x)$ is palindromic.
Combining this with Theorem~\ref{thm+minu+poly+RZ}, 
we derive from a known result~\cite[Remark 7.3.1]{Br15} the following $\gamma$-positivity, 
which gives another generalization of both the palindromicity and unimodality of $N_n(x)$~\cite[Corollaries 3.2 and  4.2]{BE23}.
\begin{coro}
There are non-negative integers $\gamma_{n,k}$ such that
\begin{equation}\label{gamma:N}
N_n(x)=\sum_{k=0}^{\lfloor n/2 \rfloor}\gamma_{n,k}x^k(1+x)^{n-2k}.
\end{equation}
\end{coro}
Below we list the above formula for
the first few values of $n$:
\begin{align*}
N_1(x)&=0,\\
N_2(x)&=2x,\\
N_3(x)&=8x(1+x),\\
N_4(x)&=20x(1+x)^2+16x^2,\\
N_5(x)&=40x(1+x)^3+96x^2(1+x),\\
N_6(x)&=70x(1+x)^4+336x^2(1+x)+16x^3.
\end{align*}
We note that the half sums of the gamma coefficients
$1, 4, 18, 68, 211, \ldots$ are not in {\sc OEIS}~\cite{slo}.
Naturally, it is interesting  to find a combinatorial interpretation of the coefficients $\gamma_{n,k}$.

\section{Asymptotic normality of the coefficients  of $N_n(x)$}\label{sec+iden+poly}

In this section, we firstly establish an identity about $N_n(x)$,
which can be viewed as a polynomial version of Conjecture \ref{conj+poly+iden}.
As an application, we show that the coefficients  of $N_n(x)$ are asymptotically normal.

\subsection{An identity related to Minuscule polynomial of type A}
\begin{prop}\label{prop+iden+minuscule}
We have
\begin{align}\label{N-formula}
   N_n(x^2)
 = &\frac{n+1}{8}\biggl((1+x)^{2n}+(1-x)^{2n}\biggr)
   -\frac{1}{16x}\biggl((1+x)^{2n+2}-(1-x)^{2n+2}\biggr).
\end{align}
\end{prop}
\begin{proof}
Although this identity can be verified by Maple using
a mechanical procedure~\cite{PWZ96}, we provide an alternative proof by manipulating binomial formulas.
From the  binomial identity
$\sum_{k=0}^{2n+2}\binom{2n+2}{k}x^k=(1+x)^{2n+2}$,
we derive
\begin{align}
    \sum_{k=0}^{n+1}\binom{2n+2}{2k}x^{2k}
 &= \frac{1}{2}\biggl((1+x)^{2n+2}+(1-x)^{2n+2}\biggr),\label{even}\\
    \sum_{k=0}^{n}\binom{2n+2}{2k+1}x^{2k+1}
 &= \frac{1}{2}\biggl((1+x)^{2n+2}-(1-x)^{2n+2}\biggr).\label{odd}
\end{align}
Differentiating \eqref{odd} we get
\begin{align}\label{diff1}
\sum_{k=0}^{n} (2k+1)\binom{2n+2}{2k+1}x^{2k}=(n+1)\biggl((1+x)^{2n+1}+(1-x)^{2n+1}\biggr).
\end{align}
Subtracting \eqref{odd} divided by $x$ from \eqref{diff1} we get
\begin{align}
   \sum_{k=1}^{n} 2k\binom{2n+2}{2k+1}x^{2k}
 = &(n+1)\biggl((1+x)^{2n+1}+(1-x)^{2n+1}\biggr)\nonumber\\
   &-\frac{1}{2x}\biggl((1+x)^{2n+2}-(1-x)^{2n+2}\biggr).\label{diff-difference}
\end{align}

Differentiating \eqref{diff1} yields
\begin{align}\label{diff2}
   \sum_{k=1}^{n} (2k+1)(2k)\binom{2n+2}{2k+1}x^{2k-1}
 = (n+1)(2n+1)\biggl((1+x)^{2n}-(1-x)^{2n}\biggr).
\end{align}
It follows from the last two identities that
\begin{multline}
   \sum_{k=1}^{n} k^2\binom{2n+2}{2k+1}x^{2k-1}
   = \frac{(n+1)(2n+1)}{4}\biggl((1+x)^{2n}-(1-x)^{2n}\biggr) \\
   -\frac{(n+1)}{4x}\biggl((1+x)^{2n+1}+(1-x)^{2n+1}\biggr) 
   +\frac{1}{8x^2}\biggl((1+x)^{2n+2}-(1-x)^{2n+2}\biggr).  \label{k-square}
\end{multline}
Combining \eqref{diff-difference} and \eqref{k-square} we get~\eqref{N-formula}.
\end{proof}

By Proposition \ref{prop+iden+minuscule}, the following result is immediate and the proof is left to the reader.

\begin{prop}
We have
$$
\sum_{n = 0}^{\infty} N_n(x^2)\frac{z^n}{n!}
=
\frac{2x(1+x)^2z-1-x^2}{16x}e^{(1+x)^2z}
+
\frac{2x(1-x)^2z+1+x^2}{16x}e^{(1-x)^2z}.
$$

\end{prop}

Setting $x=1$ in \eqref{N-formula} yields the following result,
which is Conjecture \ref{conj+poly+iden}.
\begin{coro}\label{coro+iden+minuscule}
For all positive integers $n$, we have
\begin{equation}\label{iden+conjecture}
N_n(1)=(n-1)2^{2n-3}.
\end{equation}
\end{coro}
%\begin{rem}

Let  $2^{[n]}$ denote the \emph{power set} of $[n]$.
For any pair $(X,Y)\in 2^{[n]}\times2^{[n]}$ every element of $[n]$ is exactly in one of the subsets in
$\{X\setminus Y, Y\setminus X, X\cap Y, \overline{X\cup Y}\}$, thus
\begin{subequations}
\begin{equation}\label{eq:refine+Haye}
\sum_{(X,Y) \in 2^{[n]}\times2^{[n]}}a^{|X\setminus Y|}b^{ |Y\setminus X|}c^{| X\cap Y|}d^{| \overline{X\cup Y}|}=(a+b+c+d)^n.
\end{equation}
Setting $a=b=z$ and $c=d=1$ we get
\begin{equation}\label{eq:ref-Haye}
\sum_{(X,Y) \in 2^{[n]}\times2^{[n]}}z^{|X\Theta Y|} = 2^n(1+z)^n,
\end{equation}
where $X\Theta Y:=(X \cup Y)\setminus (X \cap Y)$ is the symmetric difference.
Differentiating and setting $z=1$ yields a known result, see \cite[(26)]{LH09} and \cite[A002699]{slo},
\begin{equation}\label{eq:Haye}
S(n):=\sum_{(X,Y) \in 2^{[n]}\times2^{[n]}}|X\Theta Y| = n\cdot2^{2n-1}.
\end{equation}
\end{subequations}
Bourn and Erickson~\cite{BE23} asked for a combinatorial proof of \eqref{iden+conjecture}, i.e., $N_n(1)=S(n-1)$ using \eqref{eq:Haye}.
In this regard, it would be interesting  to find a combinatorial proof of \eqref{N-formula} in terms of $2^{[n]}\times 2^{[n]}$.
\subsection{Asymptotic normality}
A random variable $X$ is said to be {\it standard normal distributed} when
\begin{equation}\label{eq+normal+dis}
{\rm Prob} \{X \leq x\} = \frac{1}{\sqrt{2\pi}}\int_{-\infty}^{x}e^{-t^2/2}dt,
\end{equation}
denoted $X\sim \mathcal{N}(0,1)$.
Let $\mathcal{N}(x)$ be the function on the right of \eqref{eq+normal+dis}.

Let $(a(n,k))_{0\leq k\leq n}$ be a double-indexed sequence of non-negative numbers.
Recall \cite{Can15} that  the sequence $a(n,k)$ is \emph{asymptotically normal with mean $\mu_n$ and variance $\sigma_n^2$} 
provided that for the normalized probabilities
\[
p(n,k):=\frac{a(n,k)}{\sum_{k}a(n,k)}\qquad (0\leq k\leq n)
\]
we have, for each fixed $x\in \R$,
\[
\sum_{k \leq \mu_n+x\sigma_n}p(n,k) \to \mathcal{N}(x),\quad\text{as}\;n\to +\infty.
\]

The following classical result was firstly published by Bender \cite[Theorem 2]{Ben73}, although Harper \cite{Har67} already used this technique  to
show the asymptotic normality of the Stirling numbers of the second kind.
We refer the reader to the survey article \cite{Can15} and references therein for background on asymptotic normality in enumeration.

\begin{lem}\label{lem+criterion+asy+normal}
Let $A_n(x)=\sum_{k=0}^na(n,k)x^k$ be a monic polynomial whose zeros are all real and non-positive.
If $\displaystyle\lim_{n\to +\infty} \sigma_n=+\infty$,
then the numbers $a(n,k)$ are asymptotically normal with mean $\mu_n$ and variance $\sigma_n^2$.
In fact,
\begin{equation*}
\mu_n=\sum_{k=1}^n\frac{1}{1+r_k}
\quad \text{and}\quad
\sigma_n^2=\sum_{k=1}^n\frac{r_k}{(1+r_k)^2},
\end{equation*}
where $-r_k \leq 0$  for $1\leq k\leq n$ are the real zeros of $A_n(x)$.
\end{lem}

\begin{thm}
The distribution of the coefficients of $N_{n}(x)$ defined in \eqref{poly+exp+formula} is asymptotically normal as $n \to +\infty$
with mean $\mu_n=n/2$ and variance
\begin{equation}\label{variance}
\sigma_n^2=\frac{n^2-n-2}{8(n-1)}.
\end{equation}
\end{thm}
\begin{proof}
Let $p(n,k)=N_{n,k}/N_n(1)$   for $1\leq k\leq n-1$, then $\sum_{k=1}^{n-1}p(n,k)=1$.
By the palindromicity of $N_{n,k}$, we have $p(n,k)=p(n,n-k)$.
By the definition of mean,
\begin{equation}\label{double mean}
\mu_n = \sum_{k=1}^{n-1}k \cdot p(n,k)
       = \frac{1}{2}\sum_{k=1}^{n-1} (k+n-k)\cdot p(n,k)
       = \frac{n}{2}.
\end{equation}
On the other hand,
\begin{equation*}
\mu_n = \sum_{k=1}^{n-1}k\cdot p(n,k)
      = \frac{\sum_{k=1}^{n-1}kN_{n,k}}{N_n(1)}
      = \frac{N_n^{'}(1)}{N_n(1)}.
      %= \sum_{k=1}^{n-1}\frac{1}{1+r_k}.
\end{equation*}
Combining the  above formula  with \eqref{double mean} results in that
\begin{equation}\label{iden+val+mean}
\frac{(N_n(x^2))^{'}}{N_n(x^2)}\bigg|_{x=1} = 2\frac{N_n^{'}(1)}{N_n(1)}=n.
\end{equation}
If  $r_k$ are the negatives of the (real) zeros of $N_n(x)$ with
 $k\in [n-1]$, then $N_n(x)=2\binom{n+1}{3}\prod_{k=1}^{n-1}(x+r_k)$. By Lemma \ref{lem+criterion+asy+normal}, the variance $\sigma_n^2$ is given by
\begin{equation}\label{iden+val+variance}
\sigma_n^2=\sum_{k=1}^{n-1} (k-\mu_n)^2\cdot p(n,k)=\sum_{k=1}^{n-1}\frac{r_k}{(1+r_k)^2}.
\end{equation}
It is easy to verify that
\begin{align}\label{iden+variance}
   \sum_{k=1}^{n-1}\frac{4xr_k}{(x^2+r_k)^2}
 &=\frac{(N_n(x^2))^{'}}{N_n(x^2)}+x\left(\frac{(N_n(x^2))^{'}}{N_n(x^2)}\right)^{'} \nonumber \\
 &=\frac{(N_n(x^2))^{'}}{N_n(x^2)}-x\left(\frac{(N_n(x^2))^{'}}{N_n(x^2)}\right)^2+x\frac{(N_n(x^2))^{''}}{N_n(x^2)}.
\end{align}
Using the expression of $N_n(x^2)$ in Proposition \ref{prop+iden+minuscule}, for $n \geq 2$, we have (by Maple)
\begin{equation}\label{val+two+diff}
(N_n(x^2))^{''}|_{x=1}= \left(2 n^{3}-3 n^{2}+n -2\right)2^{2n-4}.
\end{equation}
Combining \eqref{iden+val+mean}--\eqref{val+two+diff} and Corollary \ref{coro+iden+minuscule}, we have
\begin{align*}
   4\sigma_n^2
 &=\left[\frac{(N_n(x^2))^{'}}{N_n(x^2)}-x\left(\frac{(N_n(x^2))^{'}}{N_n(x^2)}\right)^2+x\frac{(N_n(x^2))^{''}}{N_n(x^2)}\right]\bigg|_{x=1}\\
 &=n-n^2+\frac{2^{2n-4}\left(2 n^{3}-3 n^{2}+n -2\right)}{2^{2n-3}(n-1)}\\
 &=\frac{n^2-n-2}{2(n-1)},
\end{align*}
which implies \eqref{variance} and that $\displaystyle\lim_{n\to +\infty}\sigma_n=+\infty$.
 The proof is complete by Lemma \ref{lem+criterion+asy+normal}.
\end{proof}

\section{$x$-log-concavity of polynomials}\label{sec+log+concave}

A polynomial sequence $(f(x))_{n \geq 0}$ is said to be $x$-\emph{log-concave} if the polynomial
$$
f^2_n(x)-f_{n+1}(x)f_{n-1}(x)
$$
has only non-negative coefficients for each $n \in \mathbb{N}^+$.
The definition of $x$-log-concavity is introduced by Stanley \cite{Sag92}.
The corresponding definition of $x$-log-convexity of a polynomial sequence is that all coefficients of the polynomial $f_{n+1}(x)f_{n-1}(x)-f^2_n(x)$ are non-negative for each $n$.
The properties $x$-log-concavity and $x$-log-convexity have received much attention in combinatorics.

The first few values of $D_n(x):=N^2_n(x)-N_{n+1}(x)N_{n-1}(x)$ ($n\geq 2$)
are listed as follows:
\begin{align*}
D_2(x)&=4x^2,\\
D_3(x)&=24x^2+16x^3+24x^4,\\
D_4(x)&=80x^2+192x^3+480x^4+192x^5+80x^6,\\
D_5(x)&=200x^2+1040x^3+4280x^4+5344x^5+4280x^6+1040x^7+200x^8.
\end{align*}
Then, we have the following result.
%%%%%%%%%%%%%%%%%%%%%%%%
\begin{thm}
The polynomial sequence $(N_n(x))_{n\geq 1}$ is $x$-log-concave.
\end{thm}
\begin{proof}
By Proposition \ref{prop+iden+minuscule} it follows that
\begin{align*}
&\quad
N^2_{n+1}(x^2)-N_{n+2}(x^2)N_n(x^2) \\
&=\frac{1}{64}\biggl[(1+x)^{4n+4}+(1-x)^{4n+4}-2\biggl(1+(8n^2+16n+6)x^2+x^4\biggr)(1-x^2)^{2n}\biggr],
\end{align*}
whose verification is left to the  reader.
Therefore,
\begin{align*}
	D_{n+1}(x)
	=\frac{1}{32}\sum_{k=0}^{2n+2}\binom{4n+4}{2k}x^k-\frac{1}{32}\biggl(1+(8n^2+16n+6)x+x^2\biggr)(1-x)^{2n}.
\end{align*}
Note that the degree of $D_{n+1}(x)$ is $2n$ and $D_{n+1}(0)=0$. Let
\[
D_{n+1}(x)=\frac{1}{32}\sum_{k=1}^{2n}D_{n,k}x^k.
\]
Then,
\[
D_{n,k}=\binom{4n+4}{2k}-(-1)^k\binom{2n}{k}-(-1)^{k}\binom{2n}{k-2}+(-1)^{k}(8n^2+16n+6)\binom{2n}{k-1}.
\]
We now proceed to show that $D_{n,k}\geq 0$ by the parity of $k$. 
Let $(a)_n=a(a+1)\cdots (a+n-1)$ for $n\geq 1$ and $(a)_0=1$.
\begin{itemize}
\item  
For $1 \leq k \leq n$ we have
\begin{align}\label{ineq+even+term1}
\frac{(2k)!(2n-2k+2)!}{(2n)!}D_{n,2k}
&= \frac{(2n+1)_{2n+4}}{(2k+1)_{2k}\cdot (2n-2k+3)_{2n-2k+2}}+ E(n,k)\nonumber\\
&\geq E(n,k),
\end{align}
where
$E(n,k)=2k(2n-2k+2)(8n^2+16n+6) -2k(2k-1)-(2n-2k+2)(2n-2k+1)$. We notice that
\begin{align*}
E(n+k,k)&=32 k^{3} n +64 k^{2} n^{2}+32 k \,n^{3}+32 k^{3}+128 k^{2} n\\
& \qquad +96 k \,n^{2}+60 k^{2}+88 n k+26 k -4 n^{2} -6 n -2,
\end{align*}
which is clearly non-negative for $n\geq 0$ and $k\geq 1$.
\item 
For $1 \leq k \leq n-1$ we have
\begin{align}
\label{ineq+even+term2}
\frac{(2k+1)!(2n-2k+1)!}{(2n)!}&D_{n,2k+1}\nonumber\\
&= \frac{(2n+1)_{2n+4}}{(2k+2)_{2k+1}\cdot (2n-2k+2)_{2n-2k+1}}-F(n,k),
\end{align}
where
$F(n,k)=32(n+1)^2(n-k)k+2(n+1)(2n+1)(4n+3)$. 
We now proceed to to prove the non-negativity of the right-hand side of
\eqref{ineq+even+term2}, i.e.,
\begin{align}\label{ineq+odd+term+equi}
(2k+2)_{2k+1}\cdot (2n-2k+2)_{2n-2k+1} F(n,k) \leq (2n+1)_{2n+4}.
\end{align}
Let $g(k):=(2k+2)_{2k+1}\cdot (2n-2k+2)_{2n-2k+1}$ and 
\begin{equation}
h(k):=F(n,k)\cdot g(k)
\end{equation}
We now show that
the function $h(k)$ is bounded
by  the right-hand side of \eqref{ineq+odd+term+equi} for  $k \in [1,n-1]$.  Firstly, for $k=1$ we have
\begin{gather}
(2n+1)_{2n+4}-h(1)=
1024(n-1)n(n+1)^2(n+2)(4n-3)(2n+1)_{2n-2},
\end{gather}
which is non-negative when $n \geq 1$.
As  $h(k)=h(n-k)$ for $k \in [1,n-1]$,
it suffices to  prove that
$h(k)$ is decreasing on  $[1, \;n/2]$,
i.e., $h(k)-h(k+1) \geq 0$ for $1 \leq k \leq n/2-1$.

By the definition of $h(k)$, it is not difficult  to get the formula
\begin{align}\label{eq+diff+h}
  h(k)-h(k+1)=\frac{g(k+1)}{(4k+3)(4k+5)}\cdot32 \left[p_1(k)+(n+1)^2 p_2(k)\right],
\end{align}
where
\begin{align*}
p_1(k)&:=(n+1)(n-2k-1)(n+1)(2n+1)(4n+3),  \\
p_2(k)&:=32(n+2)k^3-48(n+2)(n-1)k^2+(16n^3-48n+62)k-15n+15.
\end{align*}
Obviously, the function $p_1(k)$ is non-negative if $1 \leq k \leq n/2-1$.
For $p_2(k)$, we have
$$
p_2^{'}(k)=96(n+2)\left[k^2-(n-1)k\right]+16n^3-48n+62.
$$
Thus, the function $p_2^{'}(k)$ in $k$ goes upwards and the symmetric axis is $(n-1)/2$.
Note that the function $p_2^{'}(k)$ takes the values
$$
p_2^{'}(1)=16n^3-96n^2-48n+446 \geq 0
$$
when $n \geq 6$, and
$$
p_2^{'}(n/2-1)=-8n^3+48n+62 \leq 0
$$
when $n \geq 3$.
Following the discussion above,
the function $p_2(k)$ firstly increases and then decreases when $1 \leq k \leq n/2-1$.
Moreover, we get that
$$
p_2(1)=(n-3)(16n^2-79)\quad \text{and} \quad p_2(n/2-1)=4n^3 - 16n - 15,
$$
which are non-negative when $n \geq 3$.
Therefore, the function $p_2(k)$ is non-negative when $1 \leq k \leq n/2-1$.
\end{itemize}
Thus, the proof is complete.
\end{proof}

Recall that a polynomial $f(x)$ is \emph{weakly Hurwitz stable} if $f(x)$ is non-vanishing or identically zero when $\Re(x) > 0$, namely all zeros of $f(x)$ lie on the closed left half-plane.
It is easy to see that the weak Hurwitz stability of polynomial $f(x) \in \mathbb{R}[x]$ implies that all coefficients of $f(x)$ have same sign.
That is, the weak Hurwitz stability is stronger than $x$-log-concavity in a sense.
Using mathematical software Matlab,
we have computed that the polynomial $N^2_{n+1}(x)-N_{n+2}(x)N_n(x)$ is weakly Hurwitz stable for $n \leq 50$.
Based on this, it is reasonable to propose the following conjecture.
\begin{conj}
The polynomial $N^2_{n+1}(x)-N_{n+2}(x)N_n(x)$ is weakly Hurwitz stable for each positive integer $n$.
\end{conj}

\section{Total positivity of the coefficient matrix}\label{sec+TP}

Let $M$ be a (finite or infinite) matrix of real numbers.
Matrix $M$ is called \emph{totally positive} ($\tp$) if all its minors are non-negative.
Let $(a_n)_{n \geq 0}$ be an infinite sequence of real numbers, and define its Toeplitz matrix as
\begin{eqnarray}
[a_{n-k}]_{n,k\geqslant0}=\left[
  \begin{array}{ccccccc}
    a_0 & \\
    a_1 & a_0 & & &\\
    a_2 & a_1 & a_0 &&\\
    a_3 & a_2 & a_1 & a_0&\\
\vdots &\vdots&\vdots&\vdots&\ddots \\
  \end{array}
\right].
\end{eqnarray}
Recall that $(a_n)_{n \geq 0}$ is said to be a \emph{P\'olya frequency} ($\pf$) sequence if its Toeplitz matrix is~$\tp$.
The following result is a classical representation theorem for $\pf$ sequences, see Karlin \cite[p.154]{Kar68}.
\begin{lem}[Schoenberg-Edrei Theorem]\label{lem+SE+PF}
A non-negative sequence $(a_n)_{n \geq 0}$ ($a_0 \neq 0$) is a $\pf$ sequence if and only if its generating function has the form
\begin{equation*}
\sum_{n=0}^{+\infty}a_nx^n=a_0 e^{\gamma x}\frac{\prod_{i=1}^{+\infty}(1+\alpha_ix)}{\prod_{i=1}^{+\infty}(1-\beta_ix)},
\end{equation*}
where $\alpha_i, \beta_i,\gamma \geq 0$ and $\sum_{i=1}^{+\infty}(\alpha_i+\beta_i) < +\infty$.
\end{lem}

%The multiplier sequence is a very useful tool to preserve the real-rootedness of polynomials under the Hadamard product.
A (infinite) sequence $\lambda=(\lambda_k)_{k\geq0}$ of real numbers is called a
\emph{multiplier sequence (of the first kind)} if a polynomial
$a_0 + a_1x + \cdots + a_n x^n$
has only real zeros, then so does the polynomial
$a_0 \lambda_0 + a_1 \lambda_1 x + \cdots + a_n \lambda_n x^n$.
The following result is the transcendental characterization for multiplier sequence.
%%%%%%%%%%%%%%%%%%%%%%%%%%%%%%%%%%
\begin{lem}\label{lem+MS+TC}\cite{PS14}
A non-negative sequence $(\lambda_n)_{n \geq 0}$ ($\lambda_0 \neq 0$) is a multiplier sequence if and only if
its generating function is a real entire function which has the form
\begin{equation*}
\sum_{n=0}^{+\infty}\frac{\lambda_n}{n!}x^n=\lambda_0e^{\gamma x}\prod_{i=1}^{+\infty}(1+\alpha_ix),
\end{equation*}
where $\gamma \geq 0, \alpha_i \geq 0$ and $\sum_{i=1}^{+\infty}\alpha_i < +\infty$.
\end{lem}

In addition, there is a classical algebraic characterization for multiplier sequence as in the following.
\begin{lem}\label{lem+alg+MS}\cite[p.100]{PS14}
	A non-negative sequence $(\lambda_k)_{k \geq 0}$ is a multiplier sequence if and only if
	the polynomial
	\begin{equation*}
		\sum_{k=0}^n\binom{n}{k}\lambda_kx^k
	\end{equation*}
	has only real zeros with same sign for each $n \geq 1$.
\end{lem}

\begin{thm}
The matrix $N=[N_{n,k}]_{n\geq k\geq0}$ is $\tp$.
\end{thm}

\begin{proof}
By \eqref{poly+exp+formula}, we have
\begin{equation*}
N_{n,k} = \frac{k(n-k)}{4n+2}\binom{2n+2}{2k+1}
        = \frac{k(2n+2)!}{(4n+2)(2k+1)!}\times\frac{n-k}{(2n-2k+1)!}.
\end{equation*}
Obviously, the total positivity of matrix $N$ is equivalent to that of matrix 
$$T:=\left[\frac{n-k}{(2n-2k+1)!}\right]_{n\geq k\geq0}.$$
Therefore, it suffices to show that the sequence $\left(\frac{k}{(2k+1)!}\right)_{k\geq0}$ is $\pf$.

Combining Lemma \ref{lem+SE+PF} and \ref{lem+MS+TC}, a non-negative sequence $\left(\frac{a_k}{k!}\right)_{k\geq0}$ is $\pf$
if the sequence $(a_k)_{k \geq 0}$ is a multiplier sequence.
Then we need to prove that the sequence $\left(\frac{k\times k!}{(2k+1)!}\right)_{k\geq0}$ is a multiplier sequence.
Invoking Lemma~\ref{lem+alg+MS} it is sufficient to show that the polynomial
\begin{equation*}
g_n(x):=\sum_{k=0}^n\binom{n}{k}\frac{k\times k!}{(2k+1)!}x^k
\end{equation*}
has only real zeros for each $n \in \mathbb{N}$.

Define the polynomial
\begin{equation*}
h_n(x):=\sum_{k=0}^n\frac{1}{(2k+1)!(n-k)!}x^k.
\end{equation*}
Note that the polynomial $h_n(x)$ satisfies the following recurrence relation
\begin{equation*}
(2n+2)(2n+3)h_{n+1}(x)=(4n+6+x)h_{n}(x)+4xh_{n}^{'}(x).
\end{equation*}
Then, it is easy to verify that
\begin{equation*}
e^{\frac{x}{4}+\frac{2n+1}{2}\ln(x)}(2n+2)(2n+3)h_{n+1}(x)=\left(e^{\frac{x}{4}+\frac{2n+1}{2}\ln(x)}(4xh_n(x))\right)^{'}.
\end{equation*}
By induction on $n$ and Rolle's theorem, it is easy to see that $h_n(x)$ has only real zeros for each positive integer $n$.
As $g_n(x)=n!xh_{n}^{'}(x)$,
the polynomial $g_n(x)$ is real-rooted, which implies the desired result.
\end{proof}

\section*{Acknowledgments}
The first author was supported by the \emph{China Scholarship Council} (No. 202206060088).
This work was done during his visit at Universit\'e Claude Bernard Lyon 1 in 2022-2023.

%%%%%%%%%%%%%%%%%%%%%

\end{document}